\documentclass[12pt]{amsart}
\usepackage{amssymb, amscd, amsmath, amsfonts, amsthm, mathrsfs}
\usepackage[T1]{fontenc}
\usepackage{anysize}

\usepackage[margin=2cm]{geometry}
\newcommand{\x}{\textbf}
\newcommand{\q}{\quad}
\newcommand{\qq}{\qquad}
\newcommand{\m}{\mathbb}
\newcommand{\mc}{\mathcal}
\newcommand{\p}{\prime}
\theoremstyle{plain}
\newtheorem{theorem}{Theorem}[section]
\newtheorem{lemma}[theorem]{Lemma}
\newtheorem{proposition}[theorem]{Proposition}
\newtheorem{corollary}[theorem]{Corollary}
\theoremstyle{definition}
\newtheorem{mydef}{Definition}[section] 
\newtheorem{myrem}{Remark}[section]    
\long\def\begincomment#1\endcomment{}

\begin{document}

\title{Some aspects of Hermitian Jacobi forms}

\author{SOUMYA~DAS}
\address{Harish Chandra Research Institute\\ 
         Chhatnag Road\\  
         Jhusi Allahabad 211019, India.}
\email{somu@hri.res.in}

\date{\today}

\begin{abstract}
We introduce a certain differential (heat) operator on the space of Hermitian Jacobi forms of degree $1$, show it's commutation with certain Hecke operators and use it to construct a lift of elliptic cusp forms to Hermitian Jacobi cusp forms. We construct Hermitian Jacobi forms as the image of the tensor product of two copies of Jacobi forms and also from differentiation of the variables. We determine the number of Fourier coefficients that determine a Hermitian Jacobi form and use it to embed a certain subspace of Hermitian Jacobi forms into a direct sum of modular forms for the full modular group.

\end{abstract}
\maketitle

\section{\x{Introduction}}

In the theory of Jacobi forms, one of the use of differential operators has been to produce new Jacobi forms from old or 
to construct other classes of modular forms, eg. elliptic modular forms. In this paper we introduce a certain differential operator $D_{\nu}, \nu \in \m{N}$ (Proposition~\ref{dnu}) on the space of Hermitian Jacobi forms of weight $k$ and index $m$ (denoted $J_{k,m}(\mc{O}_{K})$) of degree $1$ for the Hermitian Jacobi group over the ring of integers of the imaginary quadratic field $\m{Q}(i)$ (see Section~\ref{defs}) to construct modular forms for $SL(2,\m{Z})$. This is the analogue for the heat operator defined and studied for classical Jacobi forms by M. Eichler and D. Zagier in \cite{zagier}. 

We compute the Fourier expansion of the adjoint of $D_{\nu}$ in Section~\ref{hconstruct}. The vanishing of $J_{1,m}(\mc{O}_{K})$ for all $m$ is proved in Lemma~\ref{weight1}. Further, we define a map from the tensor product of two copies of classical Jacobi forms (denoted $J_{k,m}$) to $J_{k,m}(\mc{O}_{K})$ and construct Hermitian Jacobi forms from those of smaller weights and indices using partial derivaties with respect to the variables $z_{1}$ and $z_{2}$. 

Also, in analogy with Jacobi forms, $D_{\nu}$ commutes with $V_{l}$ ($l \in \m{N}$) operators in a certain sense (see Section~\ref{hecke}). In Section \ref{coeffs}, using the Theta Correspondence between $J_{k,m}(\mc{O}_{K})$ and modular forms on congruence subgroups of $SL(2,\m{Z}) $, we compute the number of Fourier coefficients that determine $J_{k,m}(\mc{O}_{K})$ (see Proposition~\ref{no. of coeffs}). Using this, an embedding of a certain subspace $J_{k,m}^{Spez}(\mc{O}_{K})$ (see Definition~\ref{spez}) of $J_{k,m}(\mc{O}_{K})$ into a finite direct sum of spaces of modular forms for $SL(2,\m{Z})$ using the $D_{\nu}$ maps is obtained (following \cite{zagier}). An analogous embedding of $J_{k,m}(\mc{O}_{K})$ is desirable, but one way to do that would be to prove the Hermitian Theta-Wronskian to be nowhere vanishing on the upper half plane. To the knowledge of the author, unfortunately we do not have this at present.   

\section{\x{Notations and definitions}} \label{defs} 

Let $\mathcal{H}$ be the upper half plane. Let $K = \m{Q}(i)$ and $\mathcal{O}_{K} = \m{Z}[i]$ be it's ring of integers. The Jacobi group over $\mathcal{O}_{K}$ is $\Gamma^{J}(\mathcal{O}_{K} ) = \Gamma^{1}(\mathcal{O}_{K} ) \ltimes \mathcal{O}_{K}^{2}$, where $\Gamma^{1}(\mathcal{O}_{K} ) = \left \{ \epsilon M \mid M \in SL(2,\m{Z}) , \epsilon \in \mathcal{O}^{\times}_{K} \right \}$.

The space of Hermitian Jacobi forms for $\Gamma^{J}(\mathcal{O}_{K} )$ of weight $k$ and index $m$ , where $k$ , $m$ are positive  integers, consists of holomorphic functions $\phi$ on $ \mathcal{H} \times \m{C}^{2}$ satisfying : \\
\begin{eqnarray}
\numberwithin{equation}{section}
\label{jacobi1} \hspace{0.6cm} \phi(\tau,z_{1},z_{2}) = \phi|_{k,m} \epsilon M(\tau,z_{1},z_{2}) := \epsilon^{-k} (c\tau + d)^{-k} e^{\frac{-2\pi i m c z_{1}z_{2}}{c\tau + d}} \phi \left(M\tau, \frac{\epsilon z_{1}}{c\tau + d},\frac{\bar{\epsilon} z_{2}}{c\tau + d}\right),  \\
\epsilon \in \mc{O}_{K}^{\times} \mbox{ for all }  M = \left( \begin{smallmatrix}
a & b \\
c & d  \end{smallmatrix}\right) \mbox{ in } SL(2,\m{Z}), M\tau = \frac{a \tau+b}{c \tau +b} \nonumber \\
\label{jacobi2} \phi(\tau,z_{1},z_{2}) = \phi|_{k,m}[\lambda,\mu] := e^{2 \pi i m (N(\lambda)\tau + \bar{\lambda}z_{1}+\lambda z_{2})}
\phi (\tau,z_{1} + \lambda \tau + \mu, z_{2} + \bar{\lambda} \tau + \bar{\mu}) \\
\mbox{ for all }  \lambda , \mu \mbox{ in }  \mathcal{O}_{K}, \mbox{where} \q N \colon K \rightarrow \m{Q} \q \mbox{is the norm map}. \nonumber 
\end{eqnarray}

The complex vector space of Hermitian Jacobi forms of weight $k$ and index $m$ is denoted by $J_{k,m}(\mc{O}_{K}).$ Such a form has a Fourier expansion : 
\begin{eqnarray} 
\label{fourier} \phi(\tau,z_{1},z_{2}) = \sum_{n = 0}^{\infty} \underset{\underset{nm \geq N(r)}{r \in \mathcal{O}_{K}^{\sharp}}}\sum c_{\phi}(n , r) e^{ 2 \pi i  \left( n \tau + r z_{1} + \bar{r} z_{2} \right)}    
\end{eqnarray}
where $\mathcal{O}_{K}^{\sharp} = \frac{i}{2} \mathcal{O}_{K}$ \q (the inverse different of\q  $K|\m{Q}$).

We say that $\phi$ is a Hermitian Jacobi \textit{cusp} form if $ c_{\phi}(n , r) = 0$ for $nm = N(r)$. The space of Jacobi cusp forms is denoted as $J_{k,m}^{cusp}(\mc{O}\sb{K} )$. 

Hermitian Jacobi forms have been defined and studied by Klaus Haverkamp in \cite{haverkamp/en}, \cite{haverkamp}.

We consider the power series expansion of $\phi$ around $z_{1} = z_{2} = 0$ from the Fourier expansion (\ref{fourier}) :
\begin{eqnarray} 
\label{powerseries} \phi(\tau,z_{1},z_{2}) = \sum_{\alpha \geq 0 , \beta \geq 0} \chi_{\alpha, \beta}(\tau) z_{1}^{\alpha}z_{2}^{\beta} 
\end{eqnarray}

We denote the space of modular forms of weight $k$ (resp. cusp forms of weight $k$) for $SL(2,\m{Z})$ by $M_{k}$ (resp. $S_{k}$).

\section{\x{A non-holomorphic differential operator on} $J_{k,m}(\mc{O}_{K} )$}

\begin{mydef} \label{diffop} 
Let $\phi_{0} : = \underset{\nu \geq 0 }\sum \chi_{\nu, \nu}(\tau) \left(z_{1}z_{2}\right)^{\nu} $ \q be the 'diagonal 
part' of $\phi \in J_{k,m}(\mc{O}_{K} )$ We denote the vector space of 'diagonal parts' arising from $J_{k,m}(\mc{O}_{K} )$ by $J_{k,m}^{0}(\mc{O}_{K} ) := \left\{ \phi_{0} \mid \phi \in J_{k,m}(\mc{O}_{K} ) \right\}$. We define the operator
\[ L_{k,m} := 8 \pi i m \frac{\partial}{\partial \tau} -  \frac{ (2k-2)}{z_{1}} \frac{\partial}{\partial z_{2}}  -   \frac{ (2k-2)}{z_{2}} \frac{\partial}{\partial z_{1}} -  4  \frac{\partial^{2}}{\partial z_{1} \partial z_{2}}\] 
\end{mydef}

\begin{lemma}
\label{|M inv1} Let $\phi$ be a holomorphic function on $\mathcal{H} \times \m{C}^2.$ Then,
\begin{eqnarray}
L_{k,m} (\phi |_{k,m} M) = (L_{k,m}\phi)|_{k+2,m}M , \mbox{ where } M \in SL(2,\m{R}).
\end{eqnarray}
\end{lemma}

\begin{proof} By direct calculation. 
\end{proof}

\begin{lemma} \label{M inv2} 
Consider the power series expansion of $\phi \in J_{k,m} (\mathcal{O}_{K} )$ as in~(\ref{powerseries}). Then the following are equivalent: 
\begin{align*}
(i) \q &\phi|_{k,m} M = \phi  \\
(ii) \q &\chi_{\alpha,\beta}\left(\frac{a \tau + b}{c \tau + d}\right) = (c \tau + d)^{k + \alpha + \beta} \underset{\alpha \geq \nu , \beta \geq \nu }
\sum \frac{1}{\nu !} \left( \frac{2 \pi i m c}{c \tau + d}\right)^{\nu}  \chi_{\alpha - \nu , \beta - \nu}(\tau)
\end{align*}
 where  $M = \left( \begin{smallmatrix}
a & b \\
c & d  \end{smallmatrix}\right) \in  SL(2,\m{Z})$.
\end{lemma}

\begin{proof} $\phi|_{k,m} M = \phi$
\begin{equation*}
\begin{split}
&  \Leftrightarrow (c \tau + d)^{-k}
e^{\frac{-2\pi i m c z_{1},z_{2}}{c\tau + d}} \phi \left(\frac{a \tau + b}{c \tau + d},
\frac{z_{1}}{c\tau + d},\frac{z_{2}}{c\tau + d}\right) = \phi
\\ 
& \Leftrightarrow \underset{\alpha \geq 0, \beta \geq 0}\sum
\chi_{\alpha,\beta} \left( \frac{a \tau + b}{c \tau + d} \right)
z_{1}^{\alpha}z_{2}^{\beta} = (c \tau + d)^{k + \alpha + \beta}\left(
\underset{\nu \geq 0}\sum \frac{1}{\nu !} \left( \frac{2 \pi i m c}{c
  \tau + d}\right)^{\nu}\left(z_{1}z_{2}\right)^{\nu}\right) \left(
\underset{\alpha,\beta}\sum \chi_{\alpha , \beta}(\tau) z_{1}^{\alpha}z_{2}^{\beta}\right)\\  
& \Leftrightarrow
\chi_{\alpha,\beta} \left( \frac{a \tau + b}{c \tau + d} \right) = (c
\tau + d)^{k + \alpha + \beta} \underset{\alpha \geq \nu , \beta \geq
  \nu } \sum \frac{1}{\nu !} \left( \frac{2 \pi i m c}{c \tau +
  d}\right)^{\nu} \chi_{\alpha - \nu , \beta - \nu}(\tau)
\end{split}
\end{equation*}

\end{proof}

\begin{myrem}\label{1stinv}
$\phi_{0} : = \underset{\nu \geq 0 , \nu \geq 0}\sum \chi_{\nu, \nu}(\tau) (z_{1}z_{2})^{\nu}  \q \mbox{ satisfies } \q \phi_{0}|_{k,m} M = \phi_{0},$ for all $\phi \in J_{k,m} (\mc{O}_{K})$. This follows from the above Proposition by replacing $\phi$ by it's diagonal part $\phi_{0}$ and retracing the proof from the last line.
\end{myrem} 

We denote the space of holomorphic functions on  $\mathcal{H} \times \m{C}^2$ satisfying the conditions of Lemma~\ref{M inv2} (i.e. invariant under the action of $SL(2,\m{Z})$) by $J_{k,m}^{0}$, so $J_{k,m}^{0}(\mc{O}_{K} ) \subset J_{k,m}^{0}$ from the above Remark. From the transformation~(\ref{jacobi1}) we get that $ \chi_{\alpha,\beta} = \epsilon^{k -\alpha + \beta}\chi_{\alpha,\beta}$ ($\epsilon \in \mc{O}_{K}^{\times}$). Hence $\chi_{\alpha,\alpha} \neq 0$ only when $k \equiv 0 \pmod{4}$. So from now on we assume $k \equiv 0 \pmod{4}$. We define the non-holomorphic differential operators next (for $k \equiv 0 \pmod{4}$), but they can be defined for other congruence classes of $k$ as well. See the Remark~\ref{otherclasses} at the end of this section.

\begin{mydef}
\label{themap}
For each $\nu \geq 0$, we denote by $\tilde{D}_{\nu}$ the composite map : 
\[ \tilde{D}_{\nu}: J_{k,m}^{0} \overset{L_{k,m}}\longrightarrow \cdots \overset{L_{k+ 2 \nu -2,m}}\longrightarrow J_{k + 2 \nu,m}^{0} \]
\end{mydef}

Clearly $\left( \tilde{D}_{\nu} \phi \right)|_{k+2 \nu,m}M = \tilde{D}_{\nu} \left( \phi |_{k,m}M \right) = \tilde{D}_{\nu}(\phi)$ for all $\phi \in J_{k,m}^{0} \mbox{ and }  M \in SL(2,\m{Z}).$

Let $\pi_{k,m}: J_{k,m}(\mc{O}_{K}) \longrightarrow J_{k,m}^{0} $ be the projection $ \phi \mapsto \phi_{0}$.

\begin{proposition}\label{dnu}
The composite map $ D_{\nu} \phi : = \tilde{D}_{\nu}\circ \pi_{k,m}  \phi \left( \tau , z_{1} , z_{2} \right) |_{ z_{1}=0 , z_{2}=0}$ defines a linear map from  $J_{k,m} (\mc{O}_{K} )$  to  $M_{k}$  for $\nu =0$  and to  $S_{k+2 \nu}$ for $\nu \geq 1$.

\begin{proof}
With Definition~(\ref{powerseries}) , $\chi_{0,0}(\tau)$ is a modular form for $  SL(2,\m{Z})$. Since $\tilde{D}_{\nu}$ are invariant under the action of $SL(2, \m{Z})$, we get the proposition. The assertion about cusp forms when $\nu \geq 1$ is trivial.
\end{proof}
\end{proposition}

\begin{proposition}
With the notation of definition~(\ref{themap}) , we have the expansion :
\[  \tilde{D}_{\nu} \phi = \sum_{\alpha = 0}^{\infty} \sum_{\mu = 0}^{\nu} (-4)^{\nu - \mu} (8 \pi i m)^{\mu} \binom{\nu}{\mu} \frac{\left( \alpha+ \nu - \mu  \right) !}{\alpha !} \frac{\left( \alpha + k+ 2 \nu - \mu -2 \right) ! }{\left( \alpha + k+ 2 \nu -2 \right) !} \chi^{(\mu)}_{\alpha + \nu - \mu, \alpha + \nu - \mu}(\tau)  \left( z_{1}z_{2} \right)^{\alpha} \]\\
where $g^{(\nu)}(\tau) = \left( \frac{\partial}{\partial \tau}\right)^{\nu}g(\tau)$.
\end{proposition}

\begin{proof}
The case $\nu = 1$ is easy to see. The rest is easily checked by induction. 

\end{proof}

\begin{corollary} \label{modular} For $\phi \in  J_{k,m}(\mc{O}_{K}) $ ,
\begin{equation}  D_{\nu} \phi = \nu ! \left( \sum_{\mu = 0}^{\nu} (-4)^{\nu - \mu} (8 \pi i m)^{\mu}  \frac{\left( k+ 2 \nu - \mu -2 \right) ! }{\mu ! \left(  k+  \nu -2 \right) !} \chi^{(\mu)}_{\nu - \mu,  \nu - \mu}(\tau) \right) \end{equation}
\end{corollary}

\begin{proof}
This follows by considering the $(0,0)^{th}$ coefficients in the above Proposition.
\end{proof}

\begin{myrem} \label{inverse}
Inverting the formula in Corollary~\ref{modular}, we get (letting $\xi_{\nu} = D_{\nu} \phi $)
\begin{equation} \chi_{\nu, \nu}(\tau) = \frac{1}{\nu ! 4^{\nu}}\left(  \sum_{\mu = 0}^{\nu} (-1)^{\nu - \mu} \left( 8 \pi i m \right)^{\mu} \binom{\nu}{\mu} \frac{ \left( k + 2 \nu -2 \mu -1 \right)  \left( k + \nu - \mu -2 \right) !}{\left( k + 2 \nu - \mu -1 \right)!} \xi_{\nu - \mu}^{\left( \mu \right)}(\tau) \right) \end{equation}
\end{myrem}

\begin{myrem}\label{otherclasses}
We can also define the $D_{\nu}$ maps  on the subseries of $\phi$ \[ \phi_{n}: = \underset{\nu \geq 0}\sum \, \chi_{\nu+n,\nu}(\tau) (z_{1}z_{2})^{\nu} \q \mbox{ or } \q \phi^{n}: = \underset{\nu \geq 0}\sum \, \chi_{\nu,\nu+n}(\tau) (z_{1}z_{2})^{\nu}\] It then follows from Remark~\ref{1stinv} in the same way as in the case of $\phi_{0}$ that $\phi_{n}$ (resp. $\phi^{n}$) are invariant under the action of $SL(2,\m{Z})$. So we can define $D_{\nu}$ operators on them. The formula for $D_{\nu}$ in this case is the same as in the case of $\phi_{0}$ except that $k$ replaced by $k+n$ and $\chi_{\alpha,\alpha}$ by $\chi_{\alpha+n,\alpha}$ (resp. by  $\chi_{\alpha+n,\alpha}$). For example, if $k \equiv 1,2,3 \pmod{4}$, one can define the $D_{\nu}$ operators on $\phi_{n}$ ($n \equiv 1,2,3 \pmod{4}$) or on $\phi^{n}$ (resp. $n \equiv 3,2,1 \pmod{4}$) and composing with the projection from $J_{k,m} (\mc{O}_{K}) $.
\end{myrem}

\section{\x{Construction of Hermitian Jacobi forms}} \label{hconstruct}

We need to consider $J_{k,m} (\mc{O}_{K}) $ for $k > 1$, because of the following lemma:

\begin{lemma} \label{weight1}
$J_{1,m}(\mc{O}_{K}) = 0$ for all $m \geq 1$.
\end{lemma} 

\begin{proof}
The proof is a simple application of the corresponding result for classical Jacobi forms, proved by N.P. Skoruppa(\cite{skoruppa}). Let $\phi \in J_{1,m}(\mc{O}_{K})$. \[ \mbox{Define} \q   U_{\rho} \colon J_{1,m} ( \mc{O}_{K})  \rightarrow J_{1,N(\rho)m}( \mc{O}_{K}),  \q 
\left(U_{\rho}\phi \right)(\tau,z_{1},z_{2})= \phi(\tau,\rho z_{1},\bar{\rho}z_{2}) \q (\rho \in \mc{O}_{K})\]
and $ \pi \colon J_{1,m}(\mc{O}_{K}) \rightarrow J_{1,m}, \q \left( \pi\phi \right) \left( \tau,z \right) = \phi(\tau,z,z)$. 
Since $J_{1,m} =0$, we have $\left( \pi \circ U_{\rho}\right) \phi = 0$ for each $\rho \in \mc{O}_{K}$. Hence we get from the power series expansion of $\phi$,
\[  \sum \chi_{\alpha,\beta} z_{1}^{\alpha}z_{2}^{\beta} \overset{U_{\rho}}{\longmapsto} \sum \chi_{\alpha,\beta} \rho^{\alpha} \bar{\rho}^{\beta} z_{1}^{\alpha}z_{2}^{\beta} \overset{\pi}{\longmapsto} \sum \chi_{\alpha,\beta} \rho^{\alpha} \bar{\rho}^{\beta} z^{\alpha+\beta}=0 \]

This clearly implies $\chi_{0,0} \equiv 0$. For each $n \geq 1$ we get the following equation
\begin{equation}\label{chi0}
\sum_{\alpha=0}^{n} \left( \frac{\rho}{\bar{\rho}} \right)^{\alpha}  \chi_{\alpha,n-\alpha}=0
\end{equation}

We choose $ \left\{ \rho_{0},\rho_{1}, \cdots \rho_{n} \right\} \in \mc{O}_{K}$ such that each $\rho_{i} \in \mc{O}_{K}\backslash \m{Z}$  and for each pair $(i,j) \in \left\{ 0,1, \cdots n \right\} \, \mbox{ with } i \neq j, \q \rho_{i} \bar{\rho}_{j} \neq \rho_{j} \bar{\rho}_{i}$; i.e., $\rho_{i} \bar{\rho}_{j} \in \mc{O}_{K}\backslash \m{Z}$. (For two sets $A,B$ we have used the notation $A \backslash B := \left\{ x \in A \mid x \not \in B \right\}$.)

We get a system of equations for each $n \geq 1$,
\[ M \cdot \Xi_{n} = 0, \mbox{ where } M_{\gamma,\alpha} =  \left( \frac{\rho_{\gamma}}{\bar{\rho}_{\gamma}} \right)^{\alpha}
\mbox{ and } \Xi_{n} = ( \chi_{\alpha,n-\alpha})_{0 \leq \alpha \leq n} \]
where $M= V \left(  \frac{\rho_{0}}{\bar{\rho}_{0}} , \cdots  \frac{\rho_{n}}{\bar{\rho}_{n}}  \right)$, is the Vandermonde determinant which is non-zero with our choice of $\rho_{i}$'s. Hence $\Xi_{n} \equiv 0$. Since this happens for every $n$, we conclude that $\chi_{\alpha,\beta} \equiv 0$ for all $\alpha,\beta$ and so $\phi \equiv 0$.
\end{proof}

\subsection{Fourier expansion of the adjoint of $ D_{\nu}$}

Let $f \in S_{k + 2 \nu} $ and $\left(  , \right)$ be the Petersson inner product on $S_{k + 2 \nu}$. Let $\langle  ,  \rangle $ be the Petersson inner product on  $J_{k,m}^{cusp} ( \mc{O}_{K}) $ and  $D_{\nu}^{*}: S_{k + 2 \nu} \longrightarrow  J_{k,m}^{cusp} ( \mc{O}_{K}) $ be the adjoint of $D_{\nu}$ with respect to the above inner products.

\begin{theorem}\label{adjoint}
With the above notations the Fourier development of $D_{\nu}^{*} f$ is given by 
\begin{eqnarray}
D_{\nu}^{*} f \left( \tau,z_{1},z_{2} \right) = \sum_{n = 0}^{\infty} \underset{\underset{nm \geq N(r)}{r \in \mc{O}_{K}^{\sharp}}}{\sum} c_{D_{\nu}^{*} f}(n , r) e^{ 2 \pi i  \left( n \tau + r z_{1} + \bar{r} z_{2} \right)}
\q \mbox{ where } , \nonumber \\ 
\begin{split}
c_{D_{\nu}^{*} f}(n , r) =& \frac{ \nu ! (-1)^{\nu} ( 4 \pi )^{2 \nu -1} \Gamma(k + 2 \nu  -1)m^{ \nu -k +3} \left( nm - N(r) \right)^{k-2}}{ \Gamma(k -2) (k -1)^{\left( \nu \right)}} \label{adjointcoeff} \\
& \times \sum_{\lambda \in \mathfrak{O}_{K}} \frac{ a \left( mN(\lambda) + r \lambda + \bar{r} \bar{\lambda} + n  , f \right)}{ \left(mN(\lambda) + r \lambda + \bar{r} \bar{\lambda} + n \right)^{k + \nu -1}}\\
& \times \sum_{j = 0}^{\nu} \frac{ (-1)^{j} (k-1)^{\left( 2 \nu - j \right) }}{\left( \nu - j \right) !^{2} j !} \left( \frac{N(m \lambda + \bar{r})}{m \left( m N(\lambda) + r \lambda + \bar{r} \bar{\lambda} + n \right)}   \right)^{\nu - j} 
\end{split}
\end{eqnarray}
where \q $f(\tau) = \sum_{n = 1}^{\infty} a \left( n, f \right) e^{2 \pi i n \tau}$. 
\end{theorem}

For the proof we make use of the Hermitian Jacobi Poincar\'{e} series. Let $n \in \m{Z} , r \in \mc{O}_{K}^{\sharp}$, and let $\Gamma^{J}: = \Gamma^{J}(\mathcal{O}_{K} )$ and $\Gamma_{\infty}^{J}: = \left\{ \left( \left( \begin{matrix} 
1 & n \\
0 & 1
\end{matrix} \right) , (0, \mu)\right) \mid n \in \m{Z}, \mu \in \mathcal{O}_{K}  \right\} \subset \Gamma^{J} $ be the stabilizer group of the function $ e^{n,r} : = e^{2 \pi i \left( n \tau + r z_{1} + \bar{r} z_{2} \right)}$ under the action of $\Gamma^{J}$. Let $ P_{n,r}^{k,m}  \left(   n \in \m{Z} , r \in \mc{O}_{K}^{\sharp}  \right)$  be the $(n,r)$-{th} Hermitian Jacobi Poincar\'{e} series of weight $k > 4$ and index $m$ defined by 
\begin{equation}
P_{n,r}^{k,m}(\tau, z_{1}, z_{2}) = \underset{ \gamma \in  \Gamma^{J}_{\infty}\backslash \Gamma^{J}}\sum e \left( n \tau + r z_{1} + \bar{r} z_{2}  \right) \mid_{k,m} \gamma (\tau, z_{1}, z_{2}) 
\end{equation}

Like in the case of classical Jacobi forms, we have :

\begin{lemma}\label{poincare}
\begin{equation*}
\langle \phi ,  P_{n,r}^{k,m} \rangle = \lambda_{n,r}^{k,m} c_{\phi}(n,r)\q  \forall \, \phi \in J_{k,m} ( \mc{O}_{K}) 
\end{equation*}
where $ \lambda_{n,r}^{k,m} =  \frac{m^{k-4} \q \Gamma(k-2)}{\left( 4 \pi \right)^{k-3} \left( mn - N(r) \right)^{k-3}} $.

\begin{proof}
Let $dV^{J} = v^{-4} du dv dx_{1} dy_{1} dx_{2} dy_{2}$ be the invariant volume element on $  \mathcal{H} \times \m{C}^{2}$ for $ \Gamma^{J}.$ We have , by the usual un-folding argument,
\begin{equation*}
 \langle \phi ,  P_{n,r}^{k,m} \rangle = \underset{\Gamma_{\infty}^{J} \backslash \mathcal{H} \times \m{C}^{2}} \int \phi(\tau,z_{1},z_{2}) \overline{e^{n,r}( \tau,z_{1},z_{2})} \q e^{\frac{- \pi m}{v} |z_{1} - \bar{z}_{2} |^{2}} v^{k} dV^{J},
\end{equation*}  
where $\tau = u + i v$ , $z_{i} = x_{j} + i y_{j}$ , $j = 1,2.$ As a fundamental domain for the action of $\Gamma_{\infty}^{J}$ on $\mathcal{H} \times \m{C}^{2}$ we take \[ \Gamma_{\infty}^{J} \backslash \mathcal{H} \times \m{C}^{2} = \left\{0 \leq u \leq 1, 0<v , 0 \leq x_{1} \leq 1, 0 \leq y_{1} \leq 1 \right\} \] We make the substitution $\bar{z}_{2} - z_{1} = z^{\prime}.$ Noting that $\mathcal{O}_{K}^{\sharp} = \frac{i}{2} \mathcal{O}_{K},$
\begin{equation*}
\begin{split}
\langle \phi ,  P_{n,r}^{k,m} \rangle & = \underset{nm > N\left( s \right)}{\sum_{l \geq 1} \sum_{s \in \mc{O}^{\sharp}_{K}}} c_{\phi}\left( l,s \right) \underset{\Gamma_{\infty}^{J} \backslash \mathfrak{H} \times \m{C}^{2}}\int e^{-2 \pi v(l + n)} e^{2 \pi i (l-n)u} e^{ 4 \pi i \textit{Re}\left( (s-r)z_{1} \right) + \bar{s} {\bar{z}^{\prime}} - r z^{\prime} } e^{ \frac{ - \pi m}{v} | z^{\prime} |^{2}} v^{k} d{V^{\prime}}^{J} \\
& = \frac{c_{\phi}\left( n,r \right)}{m} \int_{0}^{\infty} v^{k-3}e^{- \frac{4 \pi v }{m} \left( mn  -  N(r)\right)  }  dv \\
& = c_{\phi}\left( n,r \right) \frac{m^{k-4} \q \Gamma(k-2)}{\left( 4 \pi \right)^{k-3} \left( mn - N(r) \right)^{k-3}} .
\end{split} 
\end{equation*}
\end{proof}
\end{lemma}
 
\begin{lemma}\label{derivative}
\begin{equation*}
\frac{\partial^{\alpha}}{\partial {z_{1}}^{\alpha}} \frac{\partial^{\alpha}}{\partial {z_{2}}^{\alpha}} \left( exp(az_{1} + bz_{2} + cz_{1}z_{2}) \right) |_{z_{1} = z_{2} =0} \\
 = \sum_{h = 0}^{\alpha} (ab)^{h} c^{\alpha - h} {\binom{\alpha}{h}}^{2} (\alpha - h)!\\
\end{equation*}

\begin{proof}
\begin{equation*}
\begin{split}
& \frac{\partial^{\alpha}}{\partial {z_{1}}^{\alpha}} \frac{\partial^{\alpha}}{\partial {z_{2}}^{\alpha}} \left( exp(az_{1} + bz_{2} + cz_{1}z_{2}) \right)  =  \frac{\partial^{\alpha}}{\partial {z_{1}}^{\alpha}} exp(az_{1})  \frac{\partial^{\alpha}}{\partial {z_{2}}^{\alpha}} \left( exp(b + cz_{1})z_{2} \right) \\
& = \frac{\partial^{\alpha}}{\partial {z_{1}}^{\alpha}} exp(az_{1})\left\{ (b+cz_{1})^{\alpha} exp(bz_{2}+cz_{1}z_{2}) \right\} = exp(bz_{2}) \frac{\partial^{\alpha}}{\partial {z_{1}}^{\alpha}} \left\{ (b+cz_{1})^{\alpha} exp\left( (a+cz_{2})z_{1}\right) \right\}\\
& = exp(bz_{2}) \sum_{h=0}^{\alpha} \binom{\alpha}{h} \frac{\alpha !}{(\alpha -h) !}c^{h} (b+cz_{1})^{\alpha -h} (a+cz_{2})^{\alpha-h} exp(az_{1} + cz_{1}z_{2}),\\
\end{split}
\end{equation*}
from which the lemma easily follows upon changing $h \mapsto \alpha - h$.
\end{proof}
\end{lemma}

\begin{proof}[Proof of Theorem~\ref{adjoint}]

Since the proof is quite similar to that in \cite[Theorem 1.1]{tokuno}, we will only include the results of our computation. From Lemma~(\ref{poincare}) we can compute the $(n,r)$-th Fourier coefficient of $D_{\nu}^{*}f$ as 
\begin{equation} \label{44}
\langle D_{\nu}^{*}f, P_{n,r} \rangle = \frac{c_{D_{\nu}^{*} f}(n , r) m^{k-4} \Gamma(k-2)}{(4 \pi)^{k-2} (nm - N(r))^{k-3}} = \left( f, D_{\nu}(P_{n,r}) \right) 
\end{equation}

Next, we compute $D_{\nu}(P_{n,r})$ as an infinite linear combination of elliptic Poincar\'{e} series (see equation~(\ref{Dpoincare})). We start with the definition of Hermitian Poincar\'{e} series : 
\begin{eqnarray} 
P_{n,r}(\tau,z_{1},z_{2}) = \underset{\lambda \in \mc{O}_{K}}\sum \, \underset{\gamma \in \Gamma_{\infty}\backslash \Gamma}\sum (c \tau + d)^{-k} e \left( - \frac{m c z_{1}z_{2}}{c \tau + d} + (m N(\lambda) + r \lambda + \bar{r} \bar{\lambda}+n)\gamma \tau \right. \nonumber \\
\left. + \frac{(m \bar{\lambda}+r)z_{1}+ (m \lambda+r)z_{2}}{c \tau + d} \right) 
\end{eqnarray}

By Lemma~(\ref{derivative}) with $a= \frac{2 \pi i (m \bar{\lambda}+r)}{c \tau + d}, b = \frac{2 \pi i (m \lambda+\bar{r})}{c \tau + d}, c =  \frac{-2 \pi i m c }{c \tau + d}$

we obtain a formula for $\tilde{\chi}_{\alpha,\alpha}$, where the power series expansion of $P_{n,r}$ is $ \underset{\alpha,\beta \geq 0}\sum \tilde{\chi}_{\alpha,\beta} \, z_{1}^{\alpha}z_{2}^{\beta}$. 
\begin{eqnarray}\label{45}
\begin{split}
\tilde{\chi}_{\alpha,\alpha} &= \sum_{h=0}^{\alpha} \frac{(2 \pi i )^{\alpha + h}}{(h !)^{2} (\alpha -h)!} \times \\
& \times \underset{\lambda \in \mc{O}_{K}}\sum \, \underset{\gamma \in \Gamma_{\infty}\backslash \Gamma}\sum (- m c)^{\alpha - h}(c \tau + d)^{-(k+h+\alpha)} N(m \lambda+\bar{r})^{h} e \left( m N(\lambda) + r \lambda + \bar{r} \bar{\lambda}+n  \right)
\end{split}
\end{eqnarray}

For convenience of notation, we let $T := m N(\lambda) + r \lambda + \bar{r} \bar{\lambda}+n $. 

Put $\tilde{P}(\tau)= \underset{\gamma \in \Gamma_{\infty}\backslash \Gamma}\sum  (- m c)^{\alpha - h}(c \tau + d)^{-(k+h+\alpha)} e(T)$. Then we have the following formula (\cite[p.30]{tokuno})
\begin{eqnarray} \label{46}
\begin{split}
\tilde{P}^{(\mu)}(\tau) & = \sum_{j=0}^{\mu} \binom{\mu}{j} (-1)^{\mu -j} (2 \pi i)^{j} (k+h+\alpha + j)^{(\mu - j)} \times \\
& \times \underset{\gamma \in \Gamma_{\infty}\backslash \Gamma}\sum (c \tau + d)^{-(k+h+\alpha + \mu + j)} c^{\mu -j} (- m c)^{\alpha - h} T^{j} e(T).
\end{split}
\end{eqnarray}
where for non-negative integers $a,b$ , $ a^{(b)} := \frac{(a+b-1)!}{(b-1)!} $.

From equation~(\ref{45}) and~(\ref{46}) we get the following expression for $\tilde{\chi}_{\alpha,\alpha} $ :
\begin{eqnarray}\label{47}
\begin{split}
\tilde{\chi}_{\alpha,\alpha} & = \sum_{h=0}^{\alpha} \frac{(2 \pi i )^{\alpha + h}}{(h !)^{2} (\alpha -h)!} \underset{\lambda \in \mc{O}_{K}}\sum N(m \lambda+\bar{r})^{h} \sum_{j=0}^{\mu} \binom{\mu}{j} (-1)^{\mu -j} (2 \pi i)^{j} \\
& \times (k+h+\alpha + j)^{(\mu - j)} \underset{\gamma \in \Gamma_{\infty}\backslash \Gamma}\sum  (c \tau + d)^{-(k+h+\alpha + \mu + j)} \, c^{\mu -j} (- m c)^{\alpha - h} \, T^{j} e(T)
\end{split}
\end{eqnarray}

Finally taking $\alpha = \nu  - \mu$ in equation~(\ref{47}) we arrive at the following expression for $D_{\nu}(P_{n,r})(\tau)$:
\begin{eqnarray}
\begin{split}
D_{\nu}(P_{n,r})(\tau) & = \underset{\lambda \in \mc{O}_{K}}\sum \, \underset{\gamma \in \Gamma_{\infty}\backslash \Gamma}\sum \sum_{\mu=0}^{\nu} \sum_{h=0}^{\nu - \mu} \sum_{j=0}^{\mu} (-1)^{\mu + h+j} \binom{\mu}{j}  (2 \pi i)^{\nu + h+j}  \frac{4^{\nu} \nu!}{\mu! (h !)^{2} (\nu - \mu -h)!} \\
& \times (k+h+\nu - \mu + j)^{(\mu - j)} \, \frac{(k+2 \nu - \mu -2)!}{(k+ \nu -2)!} m^{\nu -h} N(m \lambda+\bar{r})^{h}\\ 
& \times (c \tau + d)^{-(k+h+ \nu + j)} \,  c^{ \nu - h -j} \, T^{j} e(T)   
\end{split}
\end{eqnarray}

Using the identities as in \cite[p.31]{tokuno} we get the expansion of $D_{\nu}(P_{n,r})(\tau)$ in terms of the Poincar\'{e} series:
\begin{eqnarray}\label{Dpoincare}
\sum_{j=0}^{\nu} \frac{(-1)^{j} (k+ 2 \nu -j -2)! \, \nu ! \, (4 \pi)^{2 \nu}}{j! (\nu -j)!^{2} (k+\nu -2)!} \underset{\lambda \in \mc{O}_{K}}\sum N(m \lambda+\bar{r})^{\nu -j} (mT)^{j} \, P_{T}^{k+ 2 \nu}(\tau)
\end{eqnarray}
where $P_{T}^{k+ 2 \nu}$ denotes the $T$-th Poincar\'{e} series of weight $k+2 \nu$ for $SL(2,\m{Z})$ 

Using Lemma~(\ref{poincare}) and equation~(\ref{44}) and the fact that $\left( f, P_{n}^{k} \right)= \frac{a(n,f) \Gamma(k-1)}{(4 \pi n)^{k-1}}$ for \\ $f = \sum_{n=0}^{\infty} a(n,f) q^{n} \in M_{k}$ and $P_{n}^{k}$ the $n$-th Poincar\'{e} series of weight $k$, we get the desired formula~(\ref{adjointcoeff}) in Theorem~\ref{adjoint}. 
\end{proof} 

\subsection{Construction of Hermitian Jacobi forms using classical Jacobi Forms} In this section we define a map from $2$ copies of Jacobi forms to Jacobi forms for the group $SL(2, \m{Z}) \ltimes \mc{O}_{K}$, denoted by $J_{k,m}^{1}(\mc{O}_{K})$ (see \cite{raghavan}); the transformation properties for these Jacobi forms being the same as in~(\ref{jacobi1}) and~(\ref{jacobi2}) except that we take $\epsilon = 1$. Obviously $J_{k,m}(\mc{O}_{K}) \subset J_{k,m}^{1}(\mc{O}_{K})$. We then average over the units in $\mc{O}_{K}$ to get Hermitian Jacobi forms.
\begin{proposition}
Fix $k_{1}, k_{2} \in \m{N}$. Let $\phi_{j} \in J_{k_{j},m}, \, j \in \{1,2\}.$ Define \[ H \left(\phi_{1},\phi_{2}\right)(\tau,z_{1},z_{2}) = \underset{\epsilon \in \mc{O}_{K}^{\times}}\sum \phi_{1} \left(\tau, \frac{1}{2}(z_{1} + z_{2}) \right) \phi_{2} \left(\tau, \frac{i}{2}(z_{1} - z_{2}) \right) \mid_{k_{1}+k_{2}} \epsilon I.\]  
Then $H \left(\phi_{1},\phi_{2}\right) \in J_{k_{1}+k_{2},m}^{1}(\mc{O}_{K} )$. If $\phi_{i} $ are cusp forms, then so is $H \left(\phi_{1},\phi_{2}\right).$  

\begin{proof} The proof is easy, so we omit it. The assertion about cusp forms is easily checked by writing the Fourier expansion of $H \left(\phi_{1},\phi_{2}\right)$ from those of $\phi_{1}$ and $\phi_{2}$. 
\end{proof}
\end{proposition}

\begin{corollary}
$H$ is a bilinear map and hence by the above proposition, defines a unique linear map, which we still denote by $H$. Further averaging over the units of $\mc{O}_{K}$, i.e., considering the map $\Lambda \colon J_{k,m}^{1}(\mc{O}_{K}) \rightarrow J_{k,m}(\mc{O}_{K})$ given by $ \phi \mapsto \underset{\epsilon \in \mc{O}_{K}^{\times}}\sum \phi\mid_{k} \epsilon I $ we have the following lift to Hermitian Jacobi forms:
\begin{equation}
J_{k_{1},m} \otimes J_{k_{2},m} \overset{H}\longrightarrow J_{k_{1}+k_{2},m}^{1}(\mc{O}_{K}) \overset{\Lambda}\longrightarrow J_{k_{1}+k_{2},m}(\mc{O}_{K})
\end{equation}
preserving cusp forms.
\end{corollary}

\begin{myrem}
If $\phi_{1} = \underset{\mu \pmod{2m}}\sum h_{\mu}\theta_{m,\mu}(\tau,z) \in J_{k_{1},m}$ and $\phi_{2} \underset{\mu \pmod{2m}}\sum g_{\mu}\theta_{m,\mu}(\tau,z) \in J_{k_{2},m}$ be their Theta decompositions (see \cite{zagier}).Then the Theta decomposition (see Section~\ref{coeffs} for definition) of $\phi_{1} \otimes \phi_{2} \in J_{k_{1}+k_{2},m}(\mc{O}_{K})$ is given by (clearly the construction of the Theta decomposition for $J_{k_{1}+k_{2},m}^{1}(\mc{O}_{K})$ is exactly the same as for $J_{k_{1}+k_{2},m}(\mc{O}_{K})$)
\[ \phi_{1} \otimes \phi_{2}(\tau,z_{1},z_{2})= \underset{\epsilon \in \mc{O}_{K}^{\times}}\sum \epsilon^{-k_{1}-k_{2}} \underset{s \in \mathcal{O}_{K}^{\sharp}/m \mathcal{O}_{K}}\sum  h_{\small{\mathrm{Re}(s)}}(\tau) \, g_{\small{\mathrm{Im}(s)}}(\tau)  \cdot \theta^{H}_{m,\epsilon s}(\tau,z_{1},z_{2}) \]  
This follows easily from the fact that $H(\theta_{m,\mu},\theta_{m,\nu}) = \theta^{H}_{m,\frac{\mu}{2} + i \frac{\nu}{2}}$ and that $\theta^{H}_{m,s} \mid_{1,m} \epsilon I = \bar{\epsilon} \theta^{H}_{m, \epsilon s}$. It would be interesting to study this map and to find how large the image is.
\end{myrem}

\subsection{\x{Construction by differentiation}} Finally we construct Hermitian Jacobi forms from smaller weights and indices using differentiation of the variables $z_{1}, z_{2}$. This is the analogue of the corresponding construction for the classical Jacobi forms \cite[Theorem 9.5]{zagier}. For a function $\phi \colon \mc{H} \times \m{C}^{2} \rightarrow \m{C}$ we let $\phi_{(j)}:= \frac{\partial}{\partial z_{j}} \phi$ for $j =1,2$ and $\phi_{(r,s)} = \frac{\partial^{2}}{\partial z_{s}\partial z_{r}} \phi$ for $r,s =1,2$.

\begin{proposition}
Let $\phi$ and $\psi$ be Hermitian Jacobi forms of weights $k_{1}$ and $k_{2}$ and index $m_{1}$ and $m_{2}$ respectively. Then 
\begin{align*}
& (i)\q   m_{1} \phi \psi_{(2)} - m_{2} \psi \phi_{(2)} \mbox{ is a Hermitian Jacobi form of weight } k_{1}+k_{2}+1 \mbox{ and index } m_{1}+ m_{2}.\\
& (ii)\q  \Big( m_{1} \phi \psi_{(1)} - m_{2} \psi \phi_{(1)} \Big)^{2} + m_{1} \phi^{2} \Big( \psi_{(1)}^{2} - \psi \psi_{(1,1)} \Big) + m_{2} \psi^{2} \Big( \phi_{(1)}^{2} - \phi \phi_{(1,1)} \Big)
\end{align*} 
is a Hermitian Jacobi form of weight $2(k_{1}+k_{2}+1)$ and index $2(m_{1}+m_{2})$.
 
\end{proposition}

\begin{proof}
(i) A meromorphic Hermitian Jacobi form $\phi$ of weight $k$ and index $0$ is a meromorphic function  $\phi \colon \mc{H} \times \m{C}^{2} \rightarrow \m{C}$ satisfying 
\begin{align*}
& \phi \left(M\tau, \frac{\epsilon z_{1}}{c\tau + d},\frac{\bar{\epsilon} z_{2}}{c\tau + d}\right) = \epsilon^{k} (c\tau + d)^{k} \phi( \tau, z_{1}, z_{2}), \, \forall \epsilon \in \mc{O}_{K}^{\times} \q \mbox{and} \\
& \phi (\tau,z_{1} + \lambda \tau + \mu, z_{2} + \bar{\lambda} \tau + \bar{\mu}) = \phi(\tau,z_{1},z_{2})
\mbox{ for all }  \lambda , \mu \mbox{ in }  \mathcal{O}_{K}
\end{align*}

Clearly $\phi_{(2)} $ (resp. $\phi_{(1,1)} $) is a meromorphic Hermitian Jacobi form of weight $k+1$ (resp. $k+2$) and index $0$ (resp. $0$) since in our case $K = \m{Q}(i)$ the unit group is isomorphic to $\m{Z}/4 \m{Z}$. Therefore, given $\phi \in J_{k_{1},m_{1}}(\mc{O}_{K})$ and $\phi \in J_{k_{2},m_{2}}(\mc{O}_{K})$ we consider the quotient $ \Phi:= \frac{\phi^{m_{2}}}{\psi^{m_{1}}} $ which is a meromorphic Hermitian Jacobi form of weight $k_{1}m_{2}-k_{2}m_{1}$ and index $0$. Then 
\[ \Phi_{(2)} = \frac{\phi^{m_{2}-1}}{\psi^{m_{1}+1}} (m_{2}\psi \phi_{(2)} - m_{1} \phi \psi_{(2)}) \]
This proves that $m_{1} \phi \psi_{(2)} - m_{2} \psi \phi_{(2)}$ is a meromorphic Hermitian Jacobi form of weight $k_{1}+k_{2}+1$ and index $m_{1}+ m_{2}$. Holomorphicity at the cusps is easy to see by writing the Fourier expansions of $\phi$ and $\psi$. In case $k_{1}m_{2}-k_{2}m_{1} < 0$, we consider $\Phi_{(2)} = \frac{\psi^{m_{1}}}{\phi^{m_{2}}}$ to get the same result.

(ii) We calculate $\Phi_{(1,1)}$ 
\begin{align*}
\Phi_{(1,1)} &= \frac{\partial}{\partial z_{1}} \Big( \frac{\phi^{m_{2}-1}}{\psi^{m_{1}+1}} ( m_{2} \psi \phi_{(2)} - m_{1} \phi \psi_{(2)} ) \Big) \\
&= \frac{\phi^{m_{2}-2}}{\psi^{m_{1}+2}} \left\{ \Big( m_{1} \phi \psi_{(1)} - m_{2} \psi \phi_{(1)} \Big)^{2} + m_{1} \phi^{2} \Big( \psi_{(1)}^{2} - \psi \psi_{(1,1)} \Big) + m_{2} \psi^{2} \Big( \phi_{(1)}^{2} - \phi \phi_{(1,1)} \Big) \right\}
\end{align*}
The same arguments as in the proof of (i) completes the proof.
\end{proof}

\section{\x{Commutation with Hecke Operators}}\label{hecke}

\begin{mydef} \label{vl}
For $l \in \m{N} $ and $\phi \colon \mc{H} \times \m{C}^{2} \rightarrow \m{C}$ let 
\begin{equation} \label{matrixreps}
\phi|_{k,m}V_{l}(\tau,z_{1},z_{2}) := l^{k-1} \, \underset{\underset{\det \gamma = l}{\gamma \in SL(2,\m{Z})\backslash M\left(2,\m{Z}\right)}}\sum (c\tau +d)^{-k} e\left( -\frac{mlc\,z_{1}z_{2}}{c\tau +d}\right) \phi\left( \gamma \tau, \frac{lz_{1}}{c\tau +d},\frac{lz_{2}}{c\tau +d}\right).
\end{equation}
\end{mydef}
Let $\phi \in J_{k,m}(\mc{O}_{K})$. Then $\phi|_{k,m}V_{l} \in J_{k,ml}(\mc{O}_{K})$ (see \cite{haverkamp}, \cite{zagier}). We consider the Fourier development of the action of $\phi|_{k,m} V_{l} $ in the next Lemma :
\begin{lemma}
\begin{equation}
\phi|_{k,m}V_{l}(\tau,z_{1},z_{2}) = \sum_{n=1}^{\infty} \underset{\underset{nm \geq N(t)}{t \in \mathcal{O}_{K}^{\sharp}}}\sum \left( \underset{\underset{t/a \in \mathcal{O}^{\sharp}}{a|\left(n,l\right)}}\sum a^{k-1} c\left(\frac{nl}{a^{2}},\frac{t}{a} \right) \right) e\left( n\tau + tz_{1} + \bar{t}z_{2} \right) 
\end{equation}
\end{lemma} 

\begin{proof}
The proof is standard and so we omit it. 
\end{proof}

\begin{proposition} Let $\nu \geq 0, \, \phi \in  J_{k,m}(\mc{O}_{K}), \, l \in \m{N},\, V_{l}$ as in Definition~(\ref{matrixreps}). Then 
\begin{equation}
D_{\nu} \left(\phi |_{k,m} V_{l} \right) = \left( D_{\nu} \phi \right) |_{k+2\nu} T_{l}
\end{equation}
where $T_{l}$ is the usual Hecke operator on elliptic modular forms.
\end{proposition}

\begin{proof}
From the definition of $D_{\nu}$ operators in~(\ref{dnu}) it is enough to prove that the following diagrams are commutative for all $k,m$ :
\[ \begin{CD} 
J^{0}_{k,m}(\mc{O}_{K})    @>L_{k,m}>>  J^{0}_{k+2,m}(\mc{O}_{K})\\
@VV{V_{l}}V                @VV{V_{l}}V\\
J^{0}_{k,ml}(\mc{O}_{K})    @>L_{k,m}>>  J^{0}_{k+2,ml}(\mc{O}_{K}) 
\end{CD} \qq , \qq 
 \begin{CD} 
J^{0}_{k,m}(\mc{O}_{K})    @>z_{1}=z_{2}=0>>  M_{k}\\
@VV{V_{l}}V                      @VV{T_{l}}V\\
J^{0}_{k,ml}(\mc{O}_{K})    @>z_{1}=z_{2}=0>>  M_{k}
\end{CD} \] 
The first diagram is commutative since $V_{l}$ maps $\phi_{0}$ (the diagonal part of $\phi$) to \\ $l^{\frac{k}{2}-1} \underset{M}\sum \left(\phi_{0}|M \right) \left(\tau,\sqrt{l}z_{1},\sqrt{l}z_{2}\right)$ and $L_{k,m}$ commutes with $|_{k,m}M$ ~(\ref{|M inv1}). That the second diagram is commutative follows from~(\ref{vl}) and the definition of $T_{l}.$
\end{proof}

\section{\x{Number of Fourier coefficients that determine $\phi$}} \label{coeffs}
We recall the Theta correspondence between Hermitian Jacobi forms and vector-valued modular forms (\cite{haverkamp/en},\cite{haverkamp}):  Let $\phi \in J_{k,m}(\mc{O}_{K})$ with Fourier expansion~(\ref{fourier})
\[ \phi = \sum_{n = 0}^{\infty} \underset{\underset{nm \geq N(r)}{r \in \mathcal{O}_{K}^{\sharp}}}\sum c_{\phi}(n , r) e^{ 2 \pi i  \left( n \tau + r z_{1} + \bar{r} z_{2} \right)}\]
It is known (\cite{haverkamp/en}, \cite{haverkamp}) that $c_{\phi}(n,r)$ depends only on $r \pmod{ m\mathcal{O}_{K}}  \mbox{ and }  D(n,r) = nm - N(r)$ where $N: K \rightarrow \m{Q}$ is the norm map, defining $c_{s}(L) := \left \{\begin{array}{cc} 
c_{\phi}(n,r) &  \mbox{if} \q r \equiv s \pmod{ m\mathcal{O}_{K}} \q \mbox{and} \q L = 4D(n,r) \\
0  & \mbox{otherwise} \\
\end{array} \right.$ 
where $s \in \mathcal{O}_{K}^{\sharp}/m \mathcal{O}_{K}$ and $L \in \m{Z},$ we can rewrite the Fourier expansion of $\phi$ as (Theta decomposition): \begin{align}
\phi(\tau,z_{1},z_{2}) = \underset{s \in \mathcal{O}_{K}^{\sharp}/m \mathcal{O}_{K}}\sum h_{s}(\tau) \cdot \theta^{H}_{m,s}(\tau,z_{1},z_{2}), \q \mbox{ where } \nonumber \\
 h_{s}(\tau) := \underset{N(s)+L/4 \in m \m{Z}}{\sum_{L=0}^{\infty}} c_{s}(L) e^{ \frac{2 \pi i L \tau}{4m}} \mbox{ and } \theta^{H}_{m,s}(\tau,z_{1},z_{2}):= \underset{r \equiv s(mod \, m\mathcal{O}_{K})}\sum e \left(\frac{N(r)}{m}\tau + rz_{1}+\bar{r}z_{2} \right).  \nonumber \\
\end{align}
Further if we let $\Theta^{H}_{m}(\tau,z_{1},z_{2}) := \left(  \theta^{H}_{m,s}(\tau,z_{1},z_{2}) \right)_{s \in \mathcal{O}_{K}^{\sharp}/m \mathcal{O}_{K}} \in \m{C}^{4m^{2}},$ then we have \cite[Theorem 2]{haverkamp/en}:
\begin{theorem} \label{theta} For $g = (g_{s})_{s \in \mathcal{O}_{K}^{\sharp}/m \mathcal{O}_{K}}, g_{s}: \mathcal{H} \rightarrow \m{C}$ holomorphic, the following are equivalent : \begin{align*}
& \mbox{ (i) } {^{t}\Theta^{H}_{m}} g \in J_{k,m}(\mc{O}_{K})\\
& \mbox{ (ii) } \parallel g(\tau) \parallel \mbox{ is bounded as } Im(\tau) \rightarrow \infty \q \mbox{ and } \q g|_{k-1}M = \overline{U_{m}(M)}g \\
& \mbox{ for all } M \in \Gamma_{1}(\mathcal{O}_{K}) = \left\{\epsilon M | \epsilon \in \mathcal{O}_{K}^{\times} , M \in SL(2,\m{Z}) \right\}  \end{align*}
where $\Theta^{H}_{m}|_{1,m}M = U_{m}(M) \cdot \Theta^{H}_{m}|_{1,m}$ is it's functional equation and $U_{m}:\Gamma_{1}(\mathcal{O}) \rightarrow U(4m^{2})$ is a homomorphism defined by it ($U(n)$ is the group of $n \times n$ unitary matrices).   
\end{theorem}

\begin{myrem} \label{thetacompmod} 
We have $h_{s} \in M_{k-1}\left( \Gamma(4m)\right)$. This follows from Theorem~(\ref{theta})(ii) and the fact that $\Gamma(4m) \subset Ker(U_{m})$ (see \cite{haverkamp/en},\cite{haverkamp} for a proof).
\end{myrem}

\begin{mydef}
For a positive integer $m$, define 
\begin{equation}
\kappa(k,m) = \left[ \frac{4m^{2}(k-1)}{3} \underset{p|4m}\prod \left(1- \frac{1}{p^{2}}\right)+\frac{m}{2}\right],
\end{equation}
 $\left[ \cdot \right]$ being the greatest-integer function. Note that $\kappa(k,m)$ also equals $\left[ \frac{\omega}{m} \right] $, where
\begin{equation}
 \omega :=    \left[ SL(2,\m{Z}) \colon \Gamma(4m) \right] \cdot \frac{k-1}{48}+\frac{m^{2}}{2} . \end{equation}
\end{mydef}

Let $r(n)$ denote the number of integral solutions of $x^{2}+y^{2}=n$. It is well known that $r(n) = 4 \delta(n)$ (see \cite{hardy} for instance), where $\delta(n) = \underset{d|n}\sum \left( \frac{-4}{d} \right)$, $\left( \frac{-4}{\cdot} \right)$ being the unique primitive Dirichlet character modulo $4$.

\flushleft{For $4m|l $, let $R(l) := R_{m}(l) = \underset{0 \leq n \leq \frac{l}{4m}}\sum \,\underset{0 \leq d \leq 4m n}\sum r(d)$.}

\begin{proposition}\label{no. of coeffs}
In the Fourier expansion~(\ref{fourier}) of a Hermitian Jacobi form $\phi$, suppose that $c_{\phi}(n,r) = 0$ for $0 \leq n \leq \kappa(k,m)$. Then $\phi \equiv 0$; i.e., $\phi$ ``is determined'' by the first $R(4m \,\kappa(k,m))$ of it's Fourier coefficients. 
\end{proposition}

\begin{proof}
The proof will use the Theta decomposition~(\ref{theta}). By Remark~\ref{thetacompmod}, only finitely many coefficients determine each $h_{s}.$ Namely, if the Fourier coefficients $c_{s}(L)$ of $h_{s}$ vanish for all $0 \leq L \leq \kappa(k,m)$, then $h_{s}$ itself vanish (see \cite{schoenberg} p.120).

Let $1 \leq L \leq  \left[ SL(2,\m{Z}) \colon \Gamma(4m) \right] \cdot \frac{k-1}{12} , s \in  \mathcal{O}_{K}^{\sharp}/m \mathcal{O}_{K}$. From the definition of $c_{s}(L)$ above, $c_{s}(L)= 0$ unless $ N(2s) + L \equiv 0 \pmod{4m \m{Z}}$. As a set of representatives $\mathcal{S}$ of $\mathcal{O}_{K}^{\sharp}$ in $\mathcal{O}_{K}^{\sharp}/m \mathcal{O}_{K}$ we take 
\begin{equation}\label{diffreps}
\mathcal{S} \colon \left\{ \frac{p}{2}+i \frac{q}{2} \right\}\, , \mbox{ where } (p,q) \in \left[-m , m-1\right] \times \left[-m , m-1\right] 
\end{equation}

With the above choice, note that $\underset{s \in \mathcal{S}}{max}(N(s)) = \frac{m^{2}}{2}.$ Therefore, when $N(2s)+L = 4mn \in 4m \m{Z}$, the bound on $L$ implies that $0 \leq n \leq \kappa(k,m)$. Since there are $\underset{0 \leq d \leq 4mn}\sum r(d)$ coefficients $c_{\phi}(n,r)$ for each $n \geq 0$, this proves the proposition. 
\end{proof}

\begin{mydef}[\cite{haverkamp/en},\cite{haverkamp}]\label{spez} We define the following subspace of $J_{k,m}(\mc{O}_{K})$ :
\[ J^{Spez}_{k,m}(\mc{O}_{K} ) := \left\{ \phi \in  J_{k,m}(\mc{O}_{K}) | \, c_{\phi}(n,r) \mbox{ depends only on } nm-N(r) \right\} \]
\end{mydef}

\begin{proposition}
Suppose that in the power series decomposition~(\ref{powerseries}) of $\phi \in J^{Spez}_{k,m}(\mc{O}_{K} )$, $\chi_{\nu,\nu}=0$ for all $0 \leq \nu \leq R(4m \, \kappa(k,m))$. Then $\phi \equiv 0$.
\end{proposition}

\begin{proof}
Since each $\chi_{\nu,\nu}$ is periodic in $\tau$ with period $1$, (this follows by taking $M = \left( \begin{smallmatrix}
1 & 1 \\
0 & 1  \end{smallmatrix}\right)$ in the identity~(\ref{M inv2}) ) and is holomorphic on $\mc{H}$, it has a unique Fourier expansion \[ {\nu !}^{2} \chi_{\nu,\nu} = \underset{n}\sum \left( \underset{r}\sum (-4 \pi^{2}N(r))^{\nu} c(n,r) \right) e(n \tau) \]

Define $\tilde{r}(n) := Card.\left\{ 0 \leq d \leq 4m n, \, d = \Box \right\}$, where $d=\Box$ means $d$ is a sum of two squares. Noticing that $c(n,r) = c(n,r^{\p})$ if $N(r) = N(r^{\p})$, we let $c(n,d) := c(n,r)$, if $d = N(2r)$. Also $\chi_{0,0} \equiv 0$ implies $c(0,0) = 0$. From the hypothesis we get for each $1 \leq n \leq R(4m \, \kappa(k,m))$

\begin{eqnarray*}
\underset{N(2r)=d}{\underset{1 \leq d \leq 4mn}\sum} r(d) d^{\nu}c(n,d) = 0 
\end{eqnarray*} 

For a fixed $n$ as above, considering this equation for $1 \leq \nu \leq \tilde{r}(n)$ we get a $\tilde{r}(n) \times \tilde{r}(n)$ matrix $M(n)$ such that (since $\tilde{r}(n) < R(4m \, \kappa(k,m))$).  \[ M(n) \cdot C(n) = 0,\q \mbox{ where } M(n)_{\nu,d}=r(d)d^{\nu}\q and \q C(n) = (c(n,d))_{1 \leq d \leq 4mn, d=\Box} \] 

Now $\det{M(n)} = c \cdot \det{  V \left(1,2, \cdots \tilde{r}(n)\right)}  \neq 0 $, where $c$ is a non-zero constant, and for integers $a_{1},a_{2}, \cdots a_{l}$, $V(a_{1},a_{2}, \cdots a_{l})$ is the Vandermonde determinant which is non-zero when $a_{i} \neq a_{j} \, \forall i \neq j$. Therefore, $C(n) \equiv 0$.

Doing this for each $0 \leq n \leq R(4m \, \kappa(k,m))$, we get $c(n,r)=0$, for $0 \leq n \leq R(4m \, \kappa(k,m))$, so $\phi \equiv 0$ from the previous Proposition. 
\end{proof}

\begin{theorem}
The map \, $\mc{D} \colon J^{Spez}_{k,m}(\mc{O}_{K} ) \longrightarrow M_{k} \q \underset{1 \leq \nu \leq R(4m \, \kappa(k,m))}\bigoplus S_{k+2\nu} $ \\defined by 
\begin{equation}
\mathcal{D}(\phi) = \left( D_{\nu}\phi \right)_{0 \leq \nu \leq R(4m \, \kappa(k,m))}
\end{equation}
is injective. 
\end{theorem}

\begin{proof}
If $\phi \in ker \mc{D}$, by definition of $\mc{D}$ we have $\xi_{\nu} := D_{\nu} \phi \equiv 0$ for all $\nu$ as in the Theorem. By Remark~(\ref{inverse}) we obtain $\chi_{\nu,\nu} \equiv 0$ for $0 \leq \nu \leq R(4m \, \kappa(k,m))$. Therefore $\phi \equiv 0$ by the previous Proposition. 
\end{proof}

\section{\textit{Acknowledgements}} The author is grateful to Prof. B. Ramakrishnan for his advice and several helpful discussions and wishes to thank Prof. N. P. Skoruppa for his interest and comments.

\end{document}